\documentclass{amsart}
\usepackage{amssymb}
\usepackage{fullpage}
\usepackage{stmaryrd}
\usepackage{mathrsfs}
\usepackage{leftidx}
\usepackage{latexsym,amsfonts,euscript}

\newtheorem{lem}{ \bf Lemma}[section]

\newtheorem{pro}[lem]{\bf Proposition}
\newtheorem{thm}[lem]{\bf Theorem}

\newcommand{\bF}{{\mathbf{F}}}

\newcommand{\Aut}{ {\rm Aut} }

\newcommand{\Out}{{\rm Out}}
\newcommand{\bcl}{{\operatorname{bcl}}}

\begin{document}

\title{The largest size of conjugacy class and the $p$-parts of finite groups}

\author{Guohua Qian and Yong Yang}

\address{Department of Mathematics, Changshu Institute of Technology, Changshu, JiangSu 215500, Peoples Republic of China}

\address{Department of Mathematics, Texas State University, 601 University Drive, San Marcos, TX 78666, USA}
\makeatletter
\email{ghqian2000@163.com, yang@txstate.edu}

\makeatother
\Large

\maketitle
\date{}

\begin{abstract}
Let $p$ be a prime and let $P$ be a Sylow $p$-subgroup of a finite nonabelian group $G$. Let $\bcl(G)$ be the size of the largest conjugacy class of the group $G$. We show that $|P/O_p(G)| < \bcl(G)$ if $G$ is not abelian.
\end{abstract}

\section{Introduction}

Throughout this paper, $G$ is a finite group. Let $P$ be a Sylow $p$-subgroup of a finite nonabelian group $G$, let $b(G)$ denote the largest irreducible character degree of $G$, and let $\bcl(G)$ denote the size of the largest conjugacy classes of a finite group $G$.

It is known for finite groups that $b(G)$ is connected with the structure of $G$. In ~\cite{GLUCK} Gluck proved that in all finite groups the index of the Fitting subgroup $\bF(G)$ in $G$ is bounded by a polynomial function of $b(G)$. For solvable groups, Gluck further shows that $|G : \bF(G)| \leq b(G)^{13/2}$ and conjectured that $|G : \bF(G)| \leq b(G)^2$. In ~\cite{MOWOLF}, this bound was improved to $|G:\bF(G)| \leq b(G)^3$.


When we focus on a single prime, a stronger bound can be found. In ~\cite{QIANSHI}, Qian and Shi showed that if $G$ is any finite group, then $|P/O_p(G)| < b(G)^2$ and $|P/O_p(G)| \leq b(G)$ if $P$ is abelian. Recently, the authors ~\cite{QianYang1} improved the previous result, and showed that for a finite nonabelian group $G$, $|P/O_p(G)| \leq (b(G)^p/p)^{\frac {1} {p-1}}$.

Since there is some analogy between conjugacy class sizes and character degrees of a finite group, one may ask: do there exist some similar results for conjugacy class sizes?

Inspired by the results in ~\cite{QIANSHI}, He and Shi ~\cite[Theorem A]{HESHI} showed that for any finite group $|P/O_p(G)|< \bcl(G)^2$ and $|P/O_p(G)| \leq \bcl(G)$ if $P$ is abelian. In ~\cite{LiuSong}, Liu and Song improved the previous bound by showing that $|P/O_p(G)| \leq (\bcl(G)^p/p)^{\frac {1} {p-1}}$ for a finite nonabelian group $G$. Yang ~\cite{Yanglargeconj} recently strengthened the bound to $|P/O_p(G)| \leq \bcl(G)$ when $p$ is an odd prime but not a Mersenne prime. In this paper we remove the extra conditions and show that as long as $G$ is nonabelian, then we will always have $|P/O_p(G)| < \bcl(G)$. This strengthened all the previously mentioned results in this paragraph. We also show that the bound is the best possible.

\section{Main Theorem} \label{sec:maintheorem}

While the proofs in ~\cite{HESHI,LiuSong,Yanglargeconj} mainly use the consequences of some orbit theorems of linear group actions, it seems one has to get a weaker bound or to exclude some important cases due to the limit of those orbit theorems. However, by using the consequence of the $k(GV)$ problem, we are able to achieve the best possible bound.

 The $p$-solvable case of the famous Brauer's $k(B)$ conjecture was discovered to be equivalent to the $k(GV)$ problem (Fong ~\cite{Fong}, Nagao ~\cite{Nagao}). Namely, when a finite group acts coprimely on a finite vector space $V$, the number of conjugacy classes of the group $G \ltimes V$ is less than or equal to $|V|$, the number of elements in the vector space. Thompson, Robinson, Maggard, Gluck, Schmid ~\cite{RobinsonThompson,GMRS,Schmid} and many others have contributed to this well-known problem, and the key to the solution is to study the orbit structure of the linear group actions. The following could be viewed as a generalization of a special case of the $k(GV)$ problem by Guralnick and Robinson. ~\cite{GuralnickRobinson}. The proof does not use the full strengthen of the $k(GV)$ problem, only the special case Kn\"{o}rr ~\cite{Knorr} proved a while back where the acting group is nilpotent.

\begin{lem} \label{lem2}
Let $G$ be a finite solvable group such that $G/\bF(G)$ is nilpotent, then we have $k(G) \leq |\bF(G)|$.
\end{lem}
\begin{proof}
This is ~\cite[Lemma 3]{GuralnickRobinson}.
\end{proof}

\begin{lem} \label{lem3}
Let $G$ be a finite nilpotent group that acts faithfully and coprimely on an abelian group $V$, and we consider the semidirect product $G \ltimes V$, then the largest conjugacy class size in $G \ltimes V$ is of size greater than $|G|$.
\end{lem}
\begin{proof}
By Lemma ~\ref{lem2}, we know that the number of conjugacy classes in $G \ltimes V$ is less than or equal to $|V|$. Also the identity is a conjugacy class of size $1$, and the result follows.
\end{proof}

\begin{lem}  \label{lem1}
Let $G$ be a Sylow $p$-subgroup of a permutation group of degree $n$. Then $|G| \leq 2^{n-1}$.
\end{lem}
\begin{proof}
This result is well known (cf. ~\cite[Lemma 5]{LiuSong}).
\end{proof}

\begin{pro}  \label{simple}
Let $G$ be one of the nonabelian simple groups and $P\in {\rm Syl}_p(\Aut(G))$ for some prime $p$. Then $\bcl(G)> 2|P|$.
\end{pro}
\begin{proof}
It was stated in ~\cite[Proposition 2.6]{Yanglargeconj} that $\bcl(G)\geq 2|P|$ but a close examination of the proof indeed shows that $\bcl(G)> 2|P|$.
\end{proof}

We now prove the main result.

\begin{thm} \label{thm2}
Let $p$ be a prime and let $P$ be a Sylow $p$-subgroup of a finite nonabelian group $G$. Let $\bcl(G)$ be the size of the largest conjugacy class of the group $G$. Then $|P/O_p(G)| < \bcl(G)$.
\end{thm}
\begin{proof}
We will work by induction on $|G|$.

Note that for any subgroup or quotient group $L$ of $G$, $\bcl(G)\geq \bcl(L)$. 

Clearly we may assume that $O_p(G)=1$. Assume that $\Phi(G)>1$. Since $O_p(G)=1$, we see that $\Phi(G)$ is a $p'$-group since $\Phi(G)$ is nilpotent. Let $T$ be the pre-image of $O_p(G/\Phi(G))$ in $G$. It is clear that $T=\Phi(G)Q$ where $Q$ is a Sylow $p$-subgroup of $T$. By the Frattini's argument, we have that $G=N_G(Q) T=N_G(Q) \Phi(G) Q = N_G(Q)$, and thus $Q$ is a normal subgroup of $G$. Thus we know that $Q=1$ and $O_p(G/\Phi(G))=1$. Hence we may assume that $\Phi(G)=1$.


Assume  that all minimal normal subgroups of $G$ are solvable.
Let $F$ be the Fitting subgroup of $G$.
Since $\Phi(G)=O_p(G)=1$,
$G = F \rtimes A$ is a semidirect product of an abelian $p'$-group $F$ and a group $A$.

Clearly, $C_G(F) = C_A(F) \times F$ and $C_A(F) \unlhd  G$. Since $F$ contains all the minimal normal subgroups of $G$, we conclude that $C_A(F) = 1$, and hence, $C_G(F) = F$. Let us investigate the subgroup $K = PF$.
Since $O_p(K)$ centralizes  $F$ and hence $O_p(K) \leq C_G(F)=F$,
it follows that $O_p(K)=1$.

By induction, we may assume that $G= K = PF$.
Observe that $G = PF$ is solvable and $P$ acts faithfully on the abelian $p'$-group $F$. By Lemma ~\ref{lem3}, we know the result follows.


Now we assume that $G$ has a nonsolvable minimal normal subgroup $V$.
Set $V = V_1 \times \cdots \times V_k$, where $V_1, \ldots, V_k$ are isomorphic nonabelian simple groups.
Let us investigate the subgroup $K = P(V \times C_G(V))$.

Since $V$ is a direct product of nonabelian simple groups, $O_p(V)=1$.
This implies that $V \cap O_p(K)=1$.
Since $V$ and $O_p(K)$ are both normal in $K$,
$O_p(K)$ centralizes $V$,  so $O_p(K) \leq C_G(V)$, and hence $O_p (K) \le O_p(C_G(V))$.
Since $C_G(V)$ is normal in $G$, we see that $O_p (C_G(V)) \le O_p (G) = 1$.
Thus, we conclude that $O_p(K) = 1$.
Therefore we may assume by induction that $G=P(V \times C_G(V))$.

Set $|C_G(V)|_p=p^u$, $|G/C_G(V)|_p=p^v.$

Clearly $O_p(C_G(V))=1$.
If $C_G(V)$ is not abelian,
then by induction there exists $t \in C_G(V)$ such that
$|t^{C_G(V)}| \geq |C_G(V)|_p=p^u$. If $C_G(V)$ is abelian,
then clearly and $p^u=1$. Thus in all cases, we can find and $t \in C_G(V)$ such that  $|t^{C_G(V)}| \geq |C_G(V)|_p=p^u$.

Let $x_i \in V_i$ such that $|x_i^{V_i}| = \bcl(V_i)$ and set $x = x_1 \cdots x_k$.
Clearly $x \in V$ and $|x^{V}| = \bcl(V_1)^k = \bcl(V)$.
Note that $G/(V \times C_G(V)) \leq \Out(V) \cong \Out(V_1) \wr S_k$,
$G /C_G(V) \leq \Aut(V)\cong \Aut(V_1) \wr S_k$.
By Lemma ~\ref{lem1},
we have $p^v=|G/C_G(V)|_p \leq 2^{k-1}(|\Aut(V_1)|_p)^k$.

\bigskip

By Proposition ~\ref{simple}, we have $$\bcl(V_1)> 2|\Aut(V_1)|_p,$$  thus
$$\bcl(V)> (2|\Aut(V_1)|_p)^k
\geq |S_k|_p (|\Aut(V_1)|_p)^k \geq |G/C_G(V)|_p=p^v,$$ and then
$$\bcl(G)\geq \bcl(V\times C_G(V)) \geq \bcl(V) \cdot \bcl(C_G(V))> |G|_p,$$
and we are done.
\end{proof}

\bigskip


Remark: We provide a family of examples to show that our result is the best possible. Let $G=K \rtimes V$ where $|V|$ is a Fermat prime, $|K|=|V|-1=2^n$ and $K$ acts fixed point freely on $V$. Note that $\bcl(G)=|V|$ and $|G/O_2(G)|_2=2^n=|V|-1$.



\section{Acknowledgement} \label{sec:Acknowledgement}

The project is partially supported by the NSFC (Nos: 11671063 and 11471054), the NSF of Jiangsu Province (No.BK20161265), and the Simons Collaboration Grants (No 499532).




\begin{thebibliography}{19}





\bibitem{Fong} P. Fong, `On the characters of p-solvable groups', Trans. Amer. Math. Soc. 98 (1961), 263-284.


\bibitem{GLUCK} {D. Gluck}, `The largest irreducible character degree of a finite group', {Canad. J. Math.} 37 (3) (1985), 442-451.

\bibitem{GMRS} {D. Gluck, K. Magaard, U. Riese, and P. Schmid}, `The solution of the $k(GV)$-problem', J. Algebra 279 (2004), 694-719.

\bibitem{GuralnickRobinson} R.M. Guralnick and G.R. Robinson, `On the commuting probability in finite groups'. J. Algebra 300 (2006), no. 2, 509-528.


\bibitem{HESHI} {L. He and W. Shi}, `The largest lengths of conjugacy classes and the Sylow subgroups of finite groups', {Arch. Math.} 86 (2006), 1-6.



\bibitem{Knorr} {R. Kn\"{o}rr}, `On the number of characters in a $p$-block of a $p$-solvable group', Illinois J. Math. 28 (1984)
181-210.

\bibitem{LiuSong} {Y. Liu and X. Song}, `On the largest conjugacy class length of a finite group'. Monatsh. Math. 174 (2014), no. 2, 259-264.

\bibitem{MOWOLF} {A. Moret\'o and T.R. Wolf}, `Orbit sizes, character degrees and Sylow subgroups', {Adv. in Math.}, 184 (2004), 18-36.

\bibitem{Nagao} {H. Nagao}, `On a conjecture of Brauer for $p$-solvable groups', J. Math. Osaka City Univ. 13 (1962), 35-38.

\bibitem{QIANSHI} {G. Qian and W. Shi}, `The largest character degree and the Sylow subgroups of finite groups', {J. Algebra} 277 (2004), 165-171.

\bibitem{QianYang1} {G. Qian and Y. Yang}, `The largest character degree and the Sylow subgroups of finite groups', J. Algebra Appl. (2016), 1650066.

\bibitem {RobinsonThompson} {G.R. Robinson and J.G. Thompson}, `On Brauer's $k(B)$-problem', J. Algebra 184 (1996), 1143-1160.

\bibitem{Schmid} P. Schmid, The Solution of the k(GV) Problem, Imperial College Press, 2008.






\bibitem{Yanglargeconj} {Y. Yang}, `The largest size of conjugacy class and the Sylow $p$-subgroups of finite groups', {Arch. Math.} (2017), no.1, 9-16.


\end{thebibliography}
\end{document}